\documentclass[10pt,a4paper]{amsart}

\usepackage{amssymb,amsmath,amsfonts}

\usepackage{graphicx}

\usepackage{mathpazo}
\usepackage[hyperindex,pageanchor]{hyperref}

\numberwithin{equation}{section}
\newtheorem{theorem}{Theorem}[section]

\newtheorem{proposition}[theorem]{Proposition}

\theoremstyle{definition}
\newtheorem{definition}[theorem]{Definition}
\theoremstyle{remark}
\newtheorem{remark}[theorem]{Remark}

\newcommand{\fgrade}{\operatorname{fgrade}}

\newcommand{\Spec}{\operatorname{Spec}}

\newcommand{\Ht}{\operatorname{ht}}

\newcommand{\Ext}{\operatorname{Ext}}
\newcommand{\Supp}{\operatorname{Supp}}

\newcommand{\Hom}{\operatorname{Hom}}
\newcommand{\drdepth}{\operatorname{DR-depth}}

\newcommand{\cd}{\operatorname{cd}}

\newcommand{\Fdepth}{\operatorname{F-depth}}

\newcommand{\vpl}{\operatornamewithlimits{\varprojlim}}

\newcommand{\vil}{\operatornamewithlimits{\varinjlim}}

\newcommand{\fm}{\frak{m}}

\newcommand{\fp}{\frak{p}}

\newcommand{\fn}{\frak{n}}

\begin{document}

\title[Equality of de Rham depth and formal grade]
{On the equality of de Rham depth and formal grade in characteristic zero}

\author[Eghbali]{Majid Eghbali \   }

\address{Department of Mathematics, Tafresh University, Tafresh, 39518--79611, Iran.}
\email{m.eghbali@tafreshu.ac.ir}
\email{m.eghbali@yahoo.com}

\thanks{The author is supported in part by a grant from Tafresh University}.

\subjclass[2010]{13D45, 14F40.}

\keywords{formal grade, de Rham depth.}

\begin{abstract}
Let $Y \subset \mathbb{P}^n_k$ be a nonsingular proper closed subset of projective $n$-space over a field $k$ of characteristic zero and let $I \subset R=k[x_0, \ldots, x_n]$ be the homogeneous defining ideal of $Y$. We show that in this case, the de Rham depth of $Y$ is the same as the so-called formal grade of $I$ in $R$.
\end{abstract}

\maketitle

\section{Introduction}

Let $Y \subset \mathbb{P}^n_k$ be a proper closed subset and
let $I \subset R=k[x_0, \ldots, x_n]$ be the homogeneous defining
ideal of $Y$ over a field $k$ with homogeneous maximal ideal $\fm=(x_0, \ldots, x_n)$.

In \cite{E} the author has shown that in the case of an $F$-finite regular local ring $A$ of positive characteristic the formal grade of $I$ in $A$,  $\fgrade (I,A):=\min \{i| \vpl H^i_{\fm}(R/I^t) \neq 0\}$ is the same as $\Fdepth Y$, Frobenius depth of $Y=\Spec (A/I)$. In this note we try to consider its analogue in characteristic zero case 
and show that the so called de Rham depth of $Y$,  $\drdepth Y$ is the same as  $\fgrade (I,R)$. To be more precise we show the following:

\begin{theorem}\label{1}
Let $Y \subset \mathbb{P}^n_k$ be a nonsingular proper closed subset of projective $n$-space over a field $k$ of equicharacteristic zero and let $I \subset R=k[x_0, \ldots, x_n]$ be the homogeneous defining
ideal of $Y$. Then
 $\drdepth Y = \fgrade (I,R)$.
\end{theorem}

\section{Preliminaries}

Let $Y$ be a scheme of finite type over a field $k$ of characteristic $0$ which admits an embedding as a closed subscheme
of a smooth scheme $X$. We denote by ${\Omega}^p_X$ the sheaf of $p$-differential forms on $X$ over $k$. With these notations in mind we define the de Rham complex of $X$ as the complex of sheaves of differential forms
 $${\Omega}^{\bullet}_{X}: \mathcal{O}_X \ \ \underrightarrow{d} \ \ {\Omega}^{1}_{X} \ \ \underrightarrow{d} \ \ {\Omega}^{2}_{X} \ \ \underrightarrow{d} \ \ \ldots.$$
 De Rham cohomolgy of $Y$ is defined as hypercohomology of the formal completion $\hat{\Omega}^{\bullet}_{X}$ of the de Rham complex of $X$ along $Y$:
$$H^i_{dR}(Y/k)=\mathcal{H}^i(\hat{X},\hat{\Omega}^{\bullet}_{X}),$$
see \cite{Gr66} and \cite{Ha75} for more details. If there is no ambiguity on the field $k$ we use $H^i_{dR}(Y)$ instead of $H^i_{dR}(Y/k)$.

\begin{definition}\label{0.1}
Let $A$ be a complete local ring with coefficient field $k$ of
characteristic zero. Let $\pi :R \twoheadrightarrow A$ be a surjection of
$k$-algebras where $R=k[[x_1, \ldots , x_n]]$ for some $n$ and let
$Y \hookrightarrow X$ (where $Y=\Spec A, X=\Spec R$) be the
corresponding closed immersion. Let $P \in Y$ be the closed
point. The (local) de Rham cohomology of $Y$ is defined by
$$ H^i_{P, dR}(Y)=\mathcal{H}^i_{P}(\hat{X},\hat{\Omega}^{\bullet}_{X}),\ \text{\ for\ all\ } i.$$
\end{definition}

by \cite[III. Proposition 3.1]{Ha75} we have 
$$ H^i_{P, dR}(Y) =\mathcal{H}^i_{P}(\hat{X},\hat{\Omega}^{\bullet}_{X}) \simeq H^i_{P, dR}(\Spec \hat{\mathcal{O}}_{X,P})$$ as
$k$-spaces for all $i$.

It is noteworthy to mention that $H^i_{dR}(Y)$ and $ H^i_{P, dR}(Y)$ are finite dimensional $k$-vector spaces for all $i$.

\begin{definition}\label{0.2}
Let $Y$ be a Noetherian scheme of equicharacteristic zero. The
"\emph{de Rham depth} of $Y$" which is abbreviated by $\drdepth Y$
is $\geq d$ if and only if for each (not necessarily closed) point
$y \in Y$,
$$H^i_{y,dR}(Y)=0, \  i < d- \dim \overline{\{y\}},$$
where $d$ is an integer and $\overline{\{y\}}$ denotes the closure
of $\{y\}$.
\end{definition}
Here, and by our assumption on the scheme $Y \subset \mathbb{P}^n_k$ we may suppose that $y$ is a closed point of $Y$. See \cite[pp. 340]{Og}.

We conclude this section with some review on the concept of formal grade which is the index of the minimal nonvanishing formal cohomology module, $\vpl H^i_{\fm}(R/I^t)$, where $R$ is a local ring with unique maximal ideal $\fm$ and $I$ is an ideal of $R$.
Recall that $H^i_{\fm}(-)$ is the local cohomology functor with support in $V(\fm)$. Notice that  local cohomology module of $R$ with support in $V(I)$ is defined as $H^i_{I}(R)=\vil \Ext^i_R (R/I^t,R)$. For the reqular local ring $R$, one can interpret the last non vanishing index of $H^i_{I}(R)$, which is known as cohomological dimension of $I$ in $R$, by using $\fgrade (I,R)$. We denote the cohomological dimension of $I$ in $R$ with $\cd(R,I)$, see Remark \ref{0.8} below. For more
information on this topic we refer the reader to \cite{Br-Sh}.

\section{Proof of Theorem 1.1}

\begin{remark} \label{0.8} Let $R=k[x_0, \ldots, x_n]$ be a polynomial ring over a field $k$ and $\fm=(x_0, \ldots, x_n)$ be its homogeneous maximal ideal. Note that $k = R/\fm$.  Using a suitable gonflement of $R$ (cf. \cite[Chapter IX, Appendice 2]{Bo}) one can take a faithfully flat local homomorphism $f: R \rightarrow S$ with $f(\fm)=\fn$ from $(R,\fm)$ to a regular local ring
$(S,\fn)$ such that $S/\fn$ is the algebraic closure of $k$. Because of faithfully flatness, $S$ contains a field.
Now note that by local duality \cite[11.2.5]{Br-Sh} $$ \vpl H^i_{\fn}(S/(IS)^t) \simeq \vpl \Hom_S( \Ext^{\dim S-i}(S/(IS)^t,S),E_S)$$
and 
$$\Hom_S (\vil \Ext^{\dim S-i}(S/(IS)^t,S),E_S) \simeq \Hom_S (H^{\dim S-i}_{IS}(S),E_S).$$
This shows that $\fgrade (IS,S)=\dim S- \cd (S,IS)$. On the other hand, by
 flat base change theorem \cite{Br-Sh}, $\ \cd (S,IS)= \cd (R,I)$. Hence, we conclude that $\fgrade (I,R) =\fgrade (IS,S)$. 
\end{remark}

\begin{remark} \label{0.9} Suppose that $\mathbb{Q} \subseteq k_0 \subseteq k$, where $k$ is a field extension of a field $k_0$. Let $X_0$ be a smooth scheme defined over $k_0$. Let $X:=X_0 \times_{k_0} k$, by \cite[III, Section 5]{Ha75}, it implies the natural isomorphism between algebraic de Rham cohomology groups:
$$H^i_{dR}(X/k) \simeq H^i_{dR}(X_0/k_0) \otimes_{k_0} k,\ \ i \in \mathbb{Z},$$
as $k$ is faithfully flat over $k_0$.
\end{remark}



\begin{proposition}\label{2}
Let $Y \subset \mathbb{P}^n_k$ be a proper closed subset of projective $n$-space over a field $k$ of characteristic zero and
let $I \subset R=k[x_0, \ldots, x_n]$ be the homogeneous defining
ideal of $Y$. Then
 $\drdepth Y \geq \fgrade (I,R)$.
\end{proposition}

\begin{proof}

By virtue of Remark \ref{0.8} and \ref{0.9}, we may reduce the question to the case $k=\mathbb{C}$. Put $\fgrade (I,R)=u$.
Thus, $\vpl H^j_{\fm}(R/I^t) =0$ for all $j <u$. By using local duality \cite[11.2.5]{Br-Sh} and taking account that the inverse limit
commutes with direct limit in the first place of $\Hom_R(-,-)$ we observe that $H^i_{I}(R) =0$ for all $i > \dim R -u$.
 As $\dim R -u >0$ then $H^{i}(\mathbb{A}^{n}_{\mathbb{C}}-C(Y),\mathcal{F})=0$ for all
$i \geq \dim R -u$ and all quasi-coherent sheaf $\mathcal{F}$ on $\mathbb{A}^{n}_{\mathbb{C}}-C(Y)$, where $\mathbb{A}^{n}_{\mathbb{C}}$ and $C(Y)$ are,
respectively, the affine cones over $\mathbb{P}^n_{\mathbb{C}}$ and
$Y$. Since the natural morphism $\mathbb{A}^{n}_{\mathbb{C}}-C(Y) \rightarrow \mathbb{P}^{n}_{\mathbb{C}}-Y$ is affine, it follows that $H^{i}(\mathbb{P}^{n}_{\mathbb{C}}-Y,\mathcal{F})=0$ for all $i \geq \dim R -u$ and all quasi-coherent sheaf $\mathcal{F}$ on $\mathbb{P}^{n}_{\mathbb{C}}-Y$. By virtue of \cite[Theorem
4.4]{Og} we have $\drdepth Y \geq t$, as required.
\end{proof}

\textbf{Proof of Theorem \ref{1}.}
 
The Proposition \ref{2} ensures that  $\drdepth Y \geq \fgrade (I,R)$.
To prove converse direction, note that by \cite[Theorem 2.8]{Og}, it is enough
to show that $\Supp H^i_I(R) \subseteq \{\fm \}$ for all $i > \Ht (I)$. If $\fp$ is a prime ideal not containing $I$, it is clear
that $(H^i_I(R))_{\fp}=0$. Hence, suppose that $\fp \neq \fm$ is a
prime ideal containing $I$. As $Y$ is nonsingular, the ideal
$IR_{\fp}$ is a complete intersection in $R_{\fp}$. Then, it follows
that $(H^i_I(R))_{\fp} \cong H^i_{IR_{\fp}}(R_{\fp}) =0$ for all
$i>\Ht (I)$.

\end{document}